\documentclass[11pt]{amsart}

\usepackage{amsfonts,amssymb,amsmath}
\usepackage{graphicx}
\usepackage{psfrag}
\usepackage{color,xcolor}
\usepackage{verbatim}
\usepackage{epstopdf}
\usepackage[text={6in,9in},centering]{geometry}
\newtheorem{thm}{Theorem}[section]
\newtheorem{lem}[thm]{Lemma}

\theoremstyle{remark}
\newtheorem{rem}[thm]{Remark}

\numberwithin{equation}{section}

\newcommand{\R}{{\mathbb R}}
\newcommand{\Z}{{\mathbb Z}}
\newcommand{\E}{{\mathbb E}}

\newcommand{\B}{{\mathcal B}}
\newcommand{\C}{{\mathcal C}}
\newcommand{\I}{{\mathcal I}}
\newcommand{\Ra}{{\mathcal R}}

\newcommand{\bfa}{{\mathbf a}}
\newcommand{\bfv}{{\mathbf v}}
\newcommand{\bfw}{{\mathbf w}}
\newcommand{\bfk}{{\mathbf k}}

\def\vep{\varepsilon}
\def\leb{\mathcal{L}}
\newcommand{\card}{{\rm card\,}}
\newcommand{\dist}{{\rm dist\,}}

\begin{document}

\title{Local level sets of the Takagi--van der Waerden function}

\author{Lai Jiang}
\address{Hangzhou Institute for Advanced Study, UCAS, Hangzhou 310024}
\email{jianglai@ucas.ac.cn}

\author{Ting-Ting Ying}
\address{Zhejiang University, Hangzhou 310058}
\email{12635035@zju.edu.cn}

\author{Yi-Yang Zhang}
\address{Zhejiang University, Hangzhou 310058}
\email{12635030@zju.edu.cn}

\thanks{Corresponding author: Lai Jiang}


\date{}


\begin{abstract}

In this paper, we investigate the Takagi--van der Waerden function,
$$
T_r(x) = \sum_{n=0}^{\infty} \frac{\phi(r^n x)}{r^n} ,\quad x\in [0,1], \quad r \geq 2,
$$
where $\phi(x)={\rm dist}(x,\Z)$ represents the distance from $x$ to the nearest integer.

Lagarias and Maddock~\cite{LM12} introduced the notion of local level sets for the classical Takagi function $T_2$. They proved that if $y$ is uniformly distributed over the range of $T_2$, then the expected number of local level sets contained in the level set at height $y$ equals $3/2$. We define the corresponding local level sets for every integer $r\geq2$ and prove that, if $y$ is uniformly distributed over the range of $T_r$, then the expected number of local level sets contained in the level set at height $y$ equals $(r+1)/r$ if $r$ is even, and $(r+1)/(r-1)$ if $r$ is odd.

\end{abstract}

\maketitle

\section{Introduction}
\subsection{Background}

In 1903, Takagi introduced a classical example of a continuous nowhere differentiable function, that is
$$
 T_2(x)=\sum_{n=0}^{\infty}\frac{\phi(2^n x)}{2^n}, \qquad x \in [0,1],
$$
where $\phi(x):=\dist(x,\Z)$ represents the distance from $x$ to the nearest integer. 
It is natural to generalize the function by replacing $2$ by integers $r \geq 2$, that is
\begin{equation}\label{eq:def-tr}
 T_r(x):=\sum_{n=0}^{\infty}\frac{\phi(r^n x)}{r^n}, \qquad x \in [0,1].
\end{equation}
This generalization is known as the Takagi--van der Waerden function.
For more background and history on the Takagi function and its generalizations, we refer to the survey of Allaart and
Kawamura~\cite{AK11}.

Given $y\in\R$, write
$$
 L_r(y):=\{x\in[0,1]:T_r(x)=y\}.
$$
Then $L_r(y)$ is called the level set of $T_r$ at height $y$.
Level sets can be viewed as a special case of slices, which is important in fractal geometry.
For generalized Takagi functions, connections between the Assouad dimension of the graph and the Hausdorff dimensions of its slices were established in~\cite{ABK24,Y20}.
Maddock~\cite{M10} obtained an upper bound for the box dimensions of the level sets of $T_2$, 
and de Amo et al.~\cite{ABDF11} subsequently determined their supremum. 
The dimensions of level sets of Weierstrass functions have also been studied; see~\cite{B02,BKM25} for more details.

For the classical Takagi function, Buczolich~\cite{B08} proved
that $L_2(y)$ is finite for Lebesgue almost every $y$.
Allaart~\cite{A13} studied the possible finite cardinalities
and their distribution, showing that every positive even integer
occurs and that two is the most common cardinality when the
height is chosen uniformly from the range of $T_2$.
For every integer $r\geq2$, Allaart~\cite{A14} proved that
$L_r(y)$ is finite for almost every $y$, whereas
$L_r(T_r(x))$ is uncountable for almost every $x\in[0,1)$.
He also established the singularity of the occupation measure.

Lagarias and Maddock~\cite{LM11,LM12} introduced an equivalence relation on the domain of the classical Takagi function,
whose equivalence classes are called local level sets.
These local level sets form a partition of each level set.
They proved that every local level set is either finite or a Cantor set, and that the expected number of local level sets
contained in a level set is $3/2$ when the height is chosen uniformly from the range of $T_2$.
Allaart~\cite{A12} subsequently gave a more elementary proof of this formula.
Allaart~\cite{A13S} extended this notion to signed Takagi functions and obtained a formula for the expected number of local level sets.

In this paper, we define local level sets for the Takagi--van der Waerden function $T_r$ for every integer $r\geq2$.
The definition agrees with the classical one when $r=2$,
and we prove that each local level set is contained in a level set.
We then compute the expected number of local level sets contained in a level set when the height is chosen uniformly from the range of $T_r$.

\subsection{Main result}

For any $y \in \R$, let $N_r^{\rm loc}(y)$ denote the number of local level sets contained in $L_r(y)$.  
Let $\leb$ denote Lebesgue measure on $\R$. 
We say that $y$ is uniformly distributed over an interval $J$ if for every Borel set $B\subset \R$,
$$
 \mathbb{P}(y\in B) =\frac{\leb ( B\cap J )}{\leb(J)}.
$$

Our main result is as follows.
\begin{thm}\label{thm:main}
Fix an integer $r\geq2$.  If $y$ is uniformly distributed over the range of $T_r$, then
$$
 \E \big[N_r^{\rm loc}(y)\big]=
 \begin{cases}
 \frac{r+1}{r},&r\text{ is even},\\[3pt]
 \frac{r+1}{r-1},&r\text{ is odd}.
 \end{cases}
$$
\end{thm}

The paper is organized as follows.
Section~2 introduces the concept of local level sets and word addresses, and establishes the functional identities.
Section~3 introduces word classes and establishes a one-to-one correspondence between chains and local level sets.
This correspondence yields a formula for the number of local level sets at a given height.
Section~4 derives a recursion for the sums of the lengths of the ranges associated with these classes.
Then we use it to compute the expected number of local level sets.

\section{Local level sets}

\subsection{Notations and local level sets}

In the remainder of the paper, we fix an integer $r\geq2$.  
For each $x\in(0,1)$, write its $r$-based expansion by
$$
 x=0.\vep_1\vep_2\cdots
 =\sum_{n=1}^{\infty}\frac{\vep_n}{r^n},
 \qquad \vep_n \in\{0,1,\ldots,r-1\}.
$$
If $x$ has two possible $r$-based expansions, we always choose the finite one. 
Define
\begin{equation}\label{eq:def-br}
 \B_r:=\bigcup_{n\in \Z^+} \Big\{\frac{j}{2r^n}: 0\leq  j \leq 2r^n \Big\}.
\end{equation}
Thus $\B_r$ consists of the endpoints and midpoints of all $r$-adic
intervals.

For each $x\in[0,1]$ and $n \in \Z^+$, we write
$$
 a_n(x):=
 \begin{cases}
  1,&r^{n-1} x\bmod1\in[0, \frac{1}{2}),\\
  -1,&r^{n-1}x\bmod1\in[\frac{1}{2},1).
 \end{cases}
$$
We write $ D_0(x)=0 $ and 
$$
 D_n(x) := \sum_{j=1}^n a_j(x).
$$

Fix $x,x'\in(0,1)$ with the $r$-based expansions $ x=0.\vep_1\vep_2\cdots$ and $x'=0.\vep_1'\vep_2'\cdots$, respectively. 
We write $x\sim_r x'$ if the following conditions hold for every $n\in \Z^+$:

\begin{enumerate}
\item[(C1)] $|D_n(x)|=|D_n(x')|$,
\item[(C2)] if $D_n(x)\neq0$, then either $\vep_n=\vep_n' $ or $\vep_n+\vep_n'=r-1$,
\end{enumerate}

It follows directly from conditions (C1) and (C2) that $\sim_r$ is an equivalence relation.

\begin{lem}\label{lem:sim-symmetry}

If $x\in [0,1] \setminus\B_r$, then $x\sim_r(1-x)$.
\end{lem}

\begin{proof}
Fix $x$ and assume that $x=0.\vep_1 \cdots \vep_n \cdots$ and
$1-x=0.\eta_1 \cdots \eta_n \cdots$.
Here $x \notin \B_r$ guarantees that these two expansions are infinite.
Then for each $n \in \Z^+$,
$$
	\vep_n =r-1-\eta_n.
$$
Thus condition (C2) holds automatically. 
Moreover, $a_n(1-x)=-a_n(x)$. 
This implies that $D_n(1-x)=-D_n(x)$.
Hence condition (C1) follows.
\end{proof}

\begin{lem}\label{lem:lls-well-def}
Let $x,x'\in(0,1)$.  If $x\sim_r x'$, then we have $T_r(x)=T_r(x')$.
\end{lem}

\begin{proof}
Fix $x\in(0,1)$ with $r$-based expansion $x=0.\vep_1 \vep_2 \cdots$.
Then for any $n \in \Z^+$,
\begin{equation}\label{eq:cal-phi-n}
 \frac{\phi(r^{n-1}x)}{r^{n-1}} 
 = \frac{1}{r^{n-1}} \Big( \frac{1-a_n(x)}{2} + a_n(x)\sum_{k=n}^{\infty}\frac{\vep_k}{r^{k-n+1}} \Big)
 = \frac{1-a_n(x)}{2 r^{n-1}} + \sum_{k=n}^{\infty}\frac{  a_n(x) \vep_k}{r^k}   .
\end{equation}
Since $a_n(x) = D_n(x) - D_{n-1}(x)$,
we have
$$
 \sum_{n=1}^{\infty}\frac{1 - a_n(x)}{r^{n-1}}
 = \sum_{n=1}^{\infty}  \frac{1 + D_{n-1}(x) - D_n(x)}{r^{n-1}}
 = \frac{r}{r-1} + (1-r)\sum_{n=1}^{\infty}\frac{D_n(x)}{r^n}.
$$
Notice that $D_n(x) = \sum_{i=1}^n a_i(x)$.
$$
 \sum_{n=1}^{\infty}\sum_{k=n}^{\infty} \frac{a_n(x)\vep_k}{r^k}
=  \sum_{k=1}^{\infty} \frac{\vep_k}{r^k} \sum_{n=1}^{k} a_n(x)
 = \sum_{n=1}^{\infty}\frac{\vep_nD_n(x)}{r^n}.
$$
Summing \eqref{eq:cal-phi-n} over $n$, we have
\begin{equation}\label{eq:cal-tr}
 T_r(x) = \frac{r}{2(r-1)} + \sum_{n=1}^{\infty} \frac{\big(\vep_n-(r-1)/2\big)D_n(x)}{r^n}.
\end{equation}
Moreover, \eqref{eq:cal-tr} also holds for $x'$.
Using $a_n^2=1$, we obtain
$$
 2 a_n(x) D_n(x)  
 = a_n(x) \big( D_n(x) + D_{n-1}(x) \big) +a_n(x)^2
 = D_n(x)^2-D_{n-1}(x)^2+1.
$$
Condition (C1) gives $D_n(x)^2 = D_n(x')^2$,
therefore $a_n(x)D_n(x)=a_n(x')D_n(x')$ for every $n\in\Z^+$.

Moreover, if $D_n(x)\neq0$, condition (C2) gives
$$
 a_n(x)\Big(\vep_n-\frac{r-1}{2}\Big)
 = -\Big|\vep_n-\frac{r-1}{2}\Big|
=-\Big|\vep_n'-\frac{r-1}{2}\Big|
= a_n(x')\Big(\vep_n'-\frac{r-1}{2}\Big) .
$$
Combining this with $a_n(x)D_n(x)=a_n(x')D_n(x')$ and $a_n(x)^2=a_n(x')^2=1$, we have
$$
 \Big(\vep_n-\frac{r-1}{2}\Big)D_n(x)
 =
 \Big(\vep_n'-\frac{r-1}{2}\Big)D_n(x').
$$
If $D_n(x)=0$, then $D_n(x')=0$ by condition (C1), so the same
equality also holds.
Thus the series in \eqref{eq:cal-tr} agree term by term for $x$ and $x'$.
Therefore $T_r(x)=T_r(x')$.
\end{proof}

Since $\sim_r$ is an equivalence relation, the equivalence class of $x$ is given by
$$
 L_{r,x}^{\rm loc}:= \{x'\in(0,1):x'\sim_r x\},
$$
which is called the \emph{local level set} of $x$. By Lemma~\ref{lem:lls-well-def}, each local level set is contained in a level set of $T_r$. For every $y\in\R$, we count these sets by
$$
 N_r^{\rm loc}(y) :=\card \big( \big\{L_{r,x}^{\rm loc}:x\in(0,1),\ T_r(x)=y \big\} \big).
$$
In this paper, we are going to study the property of $N_r^{\rm loc}(y)$.

\begin{rem}

For $r=2$, condition (C2) is automatic, and the definition agrees with that of Lagarias and Maddock~\cite{LM12}. 
If $r$ is odd, the central digit cylinder crosses $1/2$, so its digit does not by itself determine $a_n$. 
This is the reason for introducing half-intervals.

\end{rem}

\subsection{Word addresses}

Put $I=[0,1/2]$. For each $n\geq0$, we define
$$
 \I_n:= \Big\{ \Big[ \frac{j-1}{2r^n},\frac{j}{2r^n} \Big] : 1 \leq j \leq r^n \Big\}.
$$
Thus $\I_n$ is a partition of $I$.

For every $0\leq k<r$, define the function $S_k:I\to I$ by
\begin{equation}\label{eq:def-sk}
 S_k(u):=\frac{ \lceil \frac{k}{2} \rceil + (-1)^k u}{r}, \qquad u \in I.
\end{equation}
Thus, $S_k(I)=[k/(2r),(k+1)/(2r)]$.

Write $\Sigma:=\{0,1,\ldots,r-1\}$.
Let $\Sigma^n$ denote the set of all the words of length $n$, where $\Sigma^0:=\{\vartheta\}$
and $\vartheta$ is the empty word by convention.
Let $\Sigma^*:=\bigcup_{n=0}^\infty\Sigma^n$ be the set of finite words.
For every $n \in \Z^+$ and $\bfw=w_1\cdots w_n\in\Sigma^n$, we set
$$
 S_{\bfw}:=S_{w_1}\circ\cdots\circ S_{w_n}, \qquad I_{\bfw}:=S_{\bfw}(I).
$$
Therefore $I_{\bfw w} \subset I_\bfw$ for all $w \in \Sigma$.
Here we use the convention $S_{\vartheta}:=\operatorname{id}_I$ and $I_{\vartheta}:=I$. 
Moreover, the map $\bfw\to I_{\bfw}$ is a bijection from $\Sigma^n$ onto $\I_n$.

For any $x\in I\setminus\B_r$, $x$ is not the endpoint of $I_\bfa$ for any $ \bfa \in \Sigma^*$.
Then for any $n \in \Z^+$, there exists a unique $\bfw \in \Sigma^n$ such that $x \in I_\bfw$.
As a result, we write its unique \emph{word address} by an infinite word $ \bfw(x)=w_1\cdots w_n\cdots$,
where $x\in I_{w_1\cdots w_n}$ for every $n\in \Z^+$.


\begin{lem}\label{lem:word-letter}
Let $x\in I\setminus\B_r$ have $r$-based expansion $x= 0. \vep_1 \cdots \vep_n \cdots$ and word address $\bfw(x)=w_1\cdots w_n\cdots$.
Then for every $n\in \Z^+$, we have $(-1)^{w_n}=a_n(x)a_{n+1}(x)$.
 Moreover,
\begin{equation}\label{eq:word-letter}
 \lfloor  w_n/2 \rfloor  =\min\{\vep_n ,r-1-\vep_n \}.
\end{equation}
\end{lem}

\begin{proof}
For each $n\in\Z^+$, write $u_{n}:=\dist( r^{n-1} x, \Z) $ and $t_n:=0.\vep_n\vep_{n+1}\cdots$.
Thus
$$u_{n} =\dist( r^{n-1} x, \Z) = \dist (t_n, \Z) = \min\{t_n,1-t_n\}. $$ 
For every $k\in\Sigma$ and $v\in I$, we have $\dist ( r S_k (v),\Z)=v$ by definition.
Combining with the word address of $x$, we obtain
$u_n=S_{w_1\cdots w_{n-1}}^{-1}(x)$.
Hence
$u_n\in S_{w_n}(I)$ and $w_n=\lfloor2ru_n\rfloor$.
If $ t_n<1/2 $, then we have $a_n(x)=1$ and
$$
 w_n= \lfloor  2r u_n \rfloor  
 = \lfloor  2\vep_n +  2 t_{n+1} \rfloor 
 =  2\vep_n + \frac{1-a_{n+1}(x)}{2}  =  2\vep_n + \frac{1-a_{n+1}(x) a_{n}(x)}{2} .
$$
If $ t_n> 1/2 $, then we have $a_n(x)=-1$ and $\vep_n \geq r- 1-\vep_n$.
Hence
$$
 w_n= 
 \lfloor  2r u_n \rfloor  
 =  2(r- 1-\vep_n )+\lfloor 2 -2 t_{n+1}  \rfloor   =   2(r- 1-\vep_n )+\frac{1 +a_{n+1}(x) }{2} .
$$
In both cases, $w_n$ is even if and only if $ a_n(x)=a_{n+1}(x) $.
This implies that $(-1)^{w_n}=a_n(x)a_{n+1}(x)$. 
Moreover, the two formulas also give
$$
 \Big\lfloor \frac{w_n}{2} \Big\rfloor
= \Big\lfloor \min\{\vep_n,r-1-\vep_n\} + \frac{1-a_{n+1}(x) a_{n}(x)}{4} \Big\rfloor
 =\min\{\vep_n,r-1-\vep_n\}.
$$
This completes the proof.
\end{proof}

From Lemma~\ref{lem:word-letter}, condition (C2) is equivalent to $\lfloor w_n/2 \rfloor=\lfloor w'_n / 2\rfloor$ if $D_n(x)\neq 0$.

The affine relations between the graph of $T_r$ and its parts motivate the following auxiliary functions. For each $s\in\Z$, define
\begin{equation}\label{eq:def-fs}
 F_s(x):=T_r(x)+sx, \qquad x\in[0,1].
\end{equation}

\begin{lem}\label{lem:relation-fs}
Fix $s\in\Z$ and $0\leq k<r$.
Then for any $u\in I$, 
\begin{equation}\label{eq:relation-fs}
r  F_s(S_k(u))=
 \begin{cases}
 	 \frac{k(s+1)}{2}+ F_{s+1}(u) , \quad &\mbox{if }  k \mbox{ is even}, \\
	\frac{(k+1)(s+1)}{2}+ F_{-s-1}(u) , \quad  &\mbox{if }  k \mbox{ is odd}.
 \end{cases}
\end{equation}
\end{lem}

\begin{proof}
In this proof, write $t=\lceil k /2 \rceil$, so we have $rS_k(u) = t + (-1)^k u$.
By \eqref{eq:def-tr}, 
		$$
		T_r \big(S_k(u) \big)
		= \sum_{n=0}^{\infty}\frac{\phi(r^n S_k(u))}{r^n}
		=S_k(u)+ \frac{1}{r}\sum_{n=0}^{\infty}\frac{\phi(r^n( t + (-1)^k u))}{r^n}=S_k(u)+\frac{1}{r}T_r ( u ).
		$$
		By \eqref{eq:def-fs}, we obtain
$$
 rF_s \big( S_k(u) \big)
 =r(s+1)S_k(u)+T_r(u)
 =(s+1) t
   +(-1)^k(s+1)u+T_r(u).
$$
Notice that $T_r(u)+(s+1)u=F_{s+1}(u)$ and $T_r(u)-(s+1)u=F_{-s-1}(u)$.
Taking even $k$ gives the first case, while taking odd $k$ gives the second case.
\end{proof}

For $s=-1$, the two cases in \eqref{eq:relation-fs} reduces to
\begin{equation}\label{eq:cal-fs-0}
r F_{-1} \big( S_k(u) \big) = F_0(u), \qquad \forall 0\leq k<r.
\end{equation}
We remark that $s=-1$ is an important case in Lemma~\ref{lem:relation-fs}.

For any $s\in\Z$, we write $ \sigma_s(\vartheta):= s $. 
For any $\bfw \in \Sigma^*$ and $w \in \Sigma$, we define 
$$
 \sigma_s( \bfw w) := (-1)^{w}(\sigma_s( \bfw )+1).
$$
Then $\sigma_s$ keeps track of the parameter in $\{F_s\}_{s\in\Z}$, and \eqref{eq:relation-fs} is equivalent to
$$
r F_s(S_k(u))  = r F_s(S_k(0)) + F_{\sigma_s(k)}(u).
$$

Let $s=0$ and $x\in I\setminus\B_r$ with word address $\bfw(x)=w_1w_2\cdots$. 
Then we have
\begin{equation}\label{eq:sigma-ad}
 \sigma_0(w_1\cdots w_n)=a_{n+1}(x)D_n(x),
 \qquad n\in\Z^+.
\end{equation}
For $n=0$ we have $\sigma_0(\vartheta)=D_0(x)=0$. If \eqref{eq:sigma-ad} holds at $n-1$, then
Lemma~\ref{lem:word-letter} gives
$$
 \sigma_0(w_1\cdots w_n)
 = a_n(x)a_{n+1}(x)\big(a_n(x)D_{n-1}(x)+1\big)
 = a_{n+1}(x) \big(D_{n-1}(x)+ a_n(x)\big).
$$
Thus \eqref{eq:sigma-ad} follows by induction.

\begin{lem}\label{lem:cal-fs-relation}
For every $s\in\Z$, $n\in\Z^+$, $\bfw\in\Sigma^n$, and $u\in I$,
\begin{equation}\label{eq:cal-fs-relation}
 F_s(S_{\bfw}(u)) = F_s(S_{\bfw}(0)) + r^{-n}F_{\sigma_s(\bfw)}(u).
\end{equation}
\end{lem}

\begin{proof}
We prove by induction on $n$.  For $n=1$, 
Lemma~\ref{lem:relation-fs} gives that for any $k\in\Sigma$,
$$
 F_s(S_k(u))  = F_s(S_k(0)) + \frac{1}{r}F_{\sigma_s(k)}(u).
$$

Assume that \eqref{eq:cal-fs-relation} holds for $\bfw\in\Sigma^n$. 
Then for $k\in\Sigma$ we obtain
\begin{align*}
 F_s(S_{\bfw k}(u))
 &= F_s(S_{\bfw}(0)) + r^{-n}F_{\sigma_s(\bfw)}(S_k(u))\\
 &= F_s(S_{\bfw}(0)) +  r^{-n}F_{\sigma_s(\bfw)}(S_k(0)) + r^{-n-1}F_{\sigma_s(\bfw k)}(u)\\
 &  = F_s(S_{\bfw k}(0))  + r^{-n-1}F_{\sigma_s(\bfw k)}(u).
\end{align*}
The induction is complete.
\end{proof}

Since $T_r=F_0$, Lemma~\ref{lem:cal-fs-relation} gives that
$$
 T_r(S_{\bfw}(u)) = T_r (S_{\bfw}(0)) +r^{-n}F_{\sigma_0(\bfw)}(u) , \qquad \forall u \in I.
$$

\section{Word classes and local level sets}

We now express local equivalence in terms of finite word addresses.

\subsection{Word classes and local equivalence}

Fix $s\in\Z$ and $n\in\Z^+$.
For any $\bfw=w_1\cdots w_n$ and $\bfk=k_1\cdots k_n$ in $\Sigma^n$,
we write $\bfw\equiv_s\bfk$
if the following conditions hold for every $1 \leq i \leq n$,
\begin{enumerate}
\item[(C3)] $\sigma_s(w_1\cdots w_i) = \sigma_s(k_1\cdots k_i),$
\item[(C4)]
if $\sigma_s(w_1\cdots w_{i-1})\neq -1$, then $w_i=k_i$. Here we use $ \sigma_s(\vartheta) =s$ for $i=1$.
\end{enumerate}
Therefore the relation $\equiv_s$ is an equivalence relation on $\Sigma^n$. We denote
the collection of its equivalence classes by
$$
\C_{s,n}:=\Sigma^n/\equiv_s .
$$
An element $C\in\C_{s,n}$ is called a word class of level $n$.
Since $\sigma_s(\bfw) = \sigma_s(\bfk)$ for all $ \bfw,\bfk \in C $, we define
$
\sigma_s(C):=\sigma_s(\bfw).
$
We set $\C_{s,0}:=\{\{\vartheta\}\}$ by convention.

\begin{lem}\label{lem:class-extension}
Fix $s\in\Z$ and $n\in\Z^+$.
Let $E\in\C_{s,n}$.
If $\sigma_s(E)\neq -1$, then $\{\bfw k:\bfw\in E\} \in \C_{s,n+1}$ for every $k\in\Sigma$.
If $\sigma_s(E)=-1$, then $ \{\bfw k:\bfw\in E,\ k\in\Sigma\} \in \C_{s,n+1}$.

\end{lem}
\begin{proof}
Let $\bfw\in E$.  
Suppose first that $\sigma_s(E)\neq-1$, and fix
$k\in\Sigma$.  
If $\bfw'\in E$, then the definition of $\equiv_s$, applied
at the first $n$ positions and then at position $n+1$, gives
$
 \bfw'k\equiv_s\bfw k.
$
Conversely, suppose that $\bfv\ell\equiv_s\bfw k$.  Restricting the defining
conditions to the first $n$ positions gives $\bfv\equiv_s\bfw$, and hence
$\bfv\in E$.  Since $\sigma_s(\bfv)=\sigma_s(\bfw)\neq-1$, the condition at position $n+1$ gives $\ell=k$.  
Therefore
$$
 \{\bfw'k:\bfw'\in E\}
$$
is exactly the equivalence class of $\bfw k$.

Now suppose that $\sigma_s(E)=-1$.  For any $\bfw'\in E$ and
$k,k'\in\Sigma$, the first $n$ positions satisfy the defining conditions,
while at position $n+1$ the last letters are unrestricted and the
corresponding values of $\sigma_s$ are both $0$. 
Hence $ \bfw'k'\equiv_s\bfw k.$
Conversely, if $\bfv\ell\equiv_s\bfw k$, then restriction to the first $n$ positions gives $\bfv\equiv_s\bfw$, and hence
$\bfv\in E$.
Thus
$$
 \{\bfw'k':\bfw'\in E,\ k'\in\Sigma\}
$$
is exactly the equivalence class of $\bfw k$.
\end{proof}

By Lemma~\ref{lem:class-extension}, every
$D\in\C_{s,n+1}$ has a unique class $E\in\C_{s,n}$ containing the
length-$n$ prefixes of the words in $D$.  We call $D$ an
\emph{extension} of $E$.

Lemma~\ref{lem:class-extension} describes word classes by extending words at the last letter.
We next give another description of their structure according to the first letter, which will be useful when applying the branch identities.

\begin{lem}\label{lem:class-prefix}
Let $s\in\Z$ and $n\geq0$. If $s\neq-1$, then 
$$
\C_{s,n+1} =  \big\{ \{ k \bfw:\bfw\in C\} : k \in\Sigma, C\in\C_{(-1)^k (s+1),n}  \big\} .
$$
If $s = -1$, then 
$
\C_{-1,n+1}= \big\{   \{ k \bfw: k \in\Sigma,\ \bfw\in C\} : C\in\C_{0,n} \big\}.
$
\end{lem}

\begin{proof}
For any $k \in\Sigma$ and $\bfa\in\Sigma^*$, the recursive definition of $\sigma_s$ gives
\begin{equation}\label{eq:sigma-first}
 \sigma_s(k \bfa) =\sigma_{\sigma_s (k)}(\bfa) =\sigma_{(-1)^k(s+1)}(\bfa).
\end{equation}
Now we fix $i,j\in\Sigma$ and $\bfw=w_1\cdots w_n,\bfk =k_1 \cdots k_n \in\Sigma^n$ such that $i\bfw\equiv_s j\bfk$.

First we assume that $s\neq-1$. 
Since $\sigma_s(\vartheta) \neq -1$, the definition of $i\bfw\equiv_s j\bfk$ requires $i=j$.
By \eqref{eq:sigma-first}, for every $0\leq m\leq n$,
$$
 \sigma_{(-1)^i(s+1)}(w_1\cdots w_m)
 =\sigma_s(iw_1\cdots w_m)
 = \sigma_s(j k_1\cdots k_m)
 =\sigma_{(-1)^j (s+1)}(k_1\cdots k_m).
$$
Therefore $i\bfw\equiv_s i\bfk$ if and only if $\bfw\equiv_{(-1)^i(s+1)}\bfk$. 
This proves the first assertion.

If $s=-1$, we have $\sigma_{-1}(i)=\sigma_{-1}(j)=0$ for all $i,j\in\Sigma$.
The remaining
conditions are exactly those defining $\bfw\equiv_0\bfk$, proving
the second assertion.
\end{proof}

\begin{lem}
\label{lem:class-function}
Fix $s\in\Z$ and $n\in\Z^+$.
If $C\in\C_{s,n}$ and $\bfw,\bfw'\in C$, then
$$
 F_s\circ S_{\bfw}=F_s\circ S_{\bfw'}.
$$
\end{lem}

\begin{proof}
We prove the statement by induction on $n$.  For $n=1$, the result follows directly from the definition of $\equiv_s$ and \eqref{eq:cal-fs-0}.

Assume that the statement holds for words of length $n$.  Let
$\bfw,\bfk\in\Sigma^n$ with $\bfw\equiv_s\bfk$, and suppose that
$\bfw w\equiv_s\bfk k$ for some $w,k\in\Sigma$.  If $\sigma_s(\bfw)\neq-1$, then $w=k$, and the induction hypothesis gives
$$
 F_s\circ S_{\bfw w}=F_s\circ S_{\bfk k}.
$$
If $\sigma_s(\bfw)=-1$, Lemma~\ref{lem:cal-fs-relation} gives
$$
\begin{aligned}
 F_s(S_{\bfw w}(u))
 &=F_s(S_{\bfw}(0))+r^{-n}F_{-1}(S_w(u)),\\
 F_s(S_{\bfk k}(u))
 &=F_s(S_{\bfk}(0))+r^{-n}F_{-1}(S_k(u)).
\end{aligned}
$$
The induction hypothesis and \eqref{eq:cal-fs-0} show that the two expressions are equal.
\end{proof}

Let $x \in I\setminus\B_r$ with $\bfw(x)=w_1w_2\cdots$.
For any $n \in \Z^+$, 
let $C_n(x)$ be the unique class in $\C_{0,n}$ containing $w_1 \cdots w_n$.
By Lemma~\ref{lem:class-extension}, $C_{n+1}(x)$ is an extension of $C_n(x)$.

\begin{lem}\label{lem:class-lls-one}
Fix $x,x' \in I\setminus \B_r$. Then we have $x \sim_r x'$ if and only if
$C_n(x)=C_n(x')$ for every $n\in\Z^+$.
\end{lem}

\begin{proof}
Assume that $\bfw(x)=w_1w_2\cdots$ and $\bfw(x')=w_1'w_2'\cdots$. 
We set $\bfw_0(x)=\bfw_0(x'):=\vartheta$.
For every $ n \in \Z^+ $, write $\bfw_n(x)=w_1\cdots w_n$ and $\bfw_n(x')=w_1'\cdots w_n'$.

We first record an identity that will be used in both directions.
For each $n \in \Z^+$, we obtain by \eqref{eq:sigma-ad} that
\begin{equation}\label{eq:sigma-d-step}
 \sigma_0(\bfw_{n-1}(x))+1
 =a_n(x)D_{n-1}(x)+1 
 =a_n(x)\big(D_{n-1}(x)+a_n(x)\big) 
 =a_n(x)D_n(x).
\end{equation}
Here we use $a_n(x)^2=1$.
For $n=1$, we use $\bfw_0(x)=\vartheta$ and $D_0(x)=0$.

First we assume that $C_n(x)=C_n(x')$ for every $n\in\Z^+$. Therefore $\bfw_n(x)\equiv_0\bfw_n(x')$.
Hence condition (C3) and \eqref{eq:sigma-ad} give
$$
 |D_n(x)|
 = | a_{n+1}(x) D_n(x) |
 = |\sigma_0(\bfw_n(x))|
 = |\sigma_0(\bfw_n(x'))|
 =  | a_{n+1}(x') D_n(x') |
 = |D_n(x')|.
$$
If $D_n(x) \neq 0 $, \eqref{eq:sigma-d-step} gives $ \sigma_0(\bfw_{n-1}(x))+1 =a_n(x)D_n(x) \neq 0$.
Hence $\sigma_0(\bfw_{n-1}(x))\neq-1$.
Condition (C4) gives $w_n=w_n'$, so condition (C2) follows by \eqref{eq:word-letter}.  
Thus $x\sim_r x'$.

Conversely, we assume that $x \sim_r x'$. 
We prove by induction that $\bfw_n(x)\equiv_0\bfw_n(x')$ holds for every $n \geq 0$.
For $n=0$, we have $\bfw_0(x)= \bfw_0(x') = \vartheta$, so the assertion holds.

Now assume that the induction hypothesis holds at level $n-1$, that is,
$$
 \bfw_{n-1}(x)\equiv_0\bfw_{n-1}(x').
$$
Set $ t:=\sigma_0\big(\bfw_{n-1}(x)\big)$, then $ t  =\sigma_0\big(\bfw_{n-1}(x')\big)$.
Applying \eqref{eq:sigma-d-step} to both $x$ and $x'$ gives
$$
a_n(x) D_n(x) =  \sigma_0\big(\bfw_{n-1}(x)\big) + 1 = t+1  =  \sigma_0 \big(\bfw_{n-1}(x')\big) + 1=a_n(x') D_n(x')   .
$$
If $t =-1$, then $D_n(x)=D_n(x')=0$, and therefore
$$
 	\big| \sigma_0(\bfw_n(x)) \big| =\big| D_n(x) \big| =0 = \big| D_n(x') \big| = \big| \sigma_0(\bfw_n(x'))  \big|.
$$
Thus condition (C3) holds at position $n$. 
Since $ t=-1$, condition (C4) is automatically satisfied.
Together with the induction hypothesis, this gives
$
 \bfw_n(x)\equiv_0\bfw_n(x').
$

If $ t \neq-1$, 
we directly have $ a_n(x) D_n(x) = t +1\neq 0 $ by \eqref{eq:sigma-d-step}.
We next compare the parities of $w_n$ and $w_n'$. 
By Lemma~\ref{lem:word-letter}, $(-1)^{w_n}=a_n(x)a_{n+1}(x)$. 
 Therefore
$$
 |D_{n+1}(x)|
 =|a_n(x)D_{n+1}(x)|
 =|a_n(x)D_n(x)+a_n(x)a_{n+1}(x)|
 =|t+1+(-1)^{w_n}|.
$$
Similarly, $ |D_{n+1}(x')| = |t+1+(-1)^{w_n'}|$.
It follows from condition (C1) that $ |D_{n+1}(x)| =|D_{n+1}(x')|$.
Notice that $|t+2|  \neq |t|$, hence $w_n$ and $w_n'$ share the same parity.
On the other hand, condition (C2) and \eqref{eq:word-letter} give 
$\lfloor{w_n/2}\rfloor = \lfloor{w_n'}/{2}\rfloor$.
This implies that $|w_n-w_n'| \leq 1$, therefore $w_n=w_n'$.
Consequently
$$
 \sigma_0(\bfw_n(x)) = \sigma_0(\bfw_n(x')).
$$
Thus $\bfw_n(x)\equiv_0\bfw_n(x')$, and the induction is complete.
\end{proof}

\subsection{Counting local level sets}

For any $C\in\C_{s,n}$ and $\bfw,\bfw'\in C$, by Lemma~\ref{lem:class-function} the range $F_s(I_{\bfw})$
is independent of the choice of $\bfw\in C$. We denote it by
$$
 \Ra_s(C):=F_s(I_{\bfw})=F_s(S_{\bfw}(I)), \qquad C\in\C_{s,n},
$$
and set $\Ra_s(\{\vartheta\}):=F_s(I)$ for $n=0$.

\begin{lem}\label{lem:class-range}
Fix $s\in\Z$ and $n\in\Z^+$.
Let $E\in\C_{s,n}$.
If $\sigma_s(E)\neq -1$, then
$$
\Ra_s(E) = \bigcup_{k\in\Sigma} \Ra_s(\{\bfw k:\bfw\in E\}).
$$
If $\sigma_s(E)=-1$, then $\Ra_s(\{\bfw k:\bfw\in E,\ k\in\Sigma\}) = \Ra_s(E)$.
\end{lem}

\begin{proof}
Assume first that $\sigma_s(E)\neq-1$.  By Lemma~\ref{lem:class-extension}, for every $k\in\Sigma$,
$\{\bfw k:\bfw\in E\}$ is an element of $\C_{s,n+1}$.  Since
$$
 I_{\bfw}=\bigcup_{k\in\Sigma}I_{\bfw k},
 \qquad \bfw\in E,
$$
we have
$$
 \Ra_s(E)
 =\bigcup_{k\in\Sigma}\Ra_s(\{\bfw k:\bfw\in E\}).
$$

Now assume that $\sigma_s(E)=-1$ and put
$$
 D:=\{\bfw k:\bfw\in E,\ k\in\Sigma\}\in\C_{s,n+1}.
$$
For any $\bfw\in E$, Lemma~\ref{lem:class-function} gives
$$
 \Ra_s(E)
 =F_s(I_{\bfw})
 =\bigcup_{k\in\Sigma}F_s(I_{\bfw k})
 =\Ra_s(D).
$$
\end{proof}

Fix $s\in\Z$ and $n\geq0$.
For $y\in\R$, let $\C_{s,n}(y)$ consist of the classes whose ranges
contain $y$:
$$
 \C_{s,n}(y):=\{C\in\C_{s,n}:y\in\Ra_s(C)\}.
$$
If for every $n\in\Z^+$,
$ E_n\in\C_{0,n}(y)$ and $E_{n+1}$ is an extension of $E_n$,
then $(E_n)_{n\in\Z^+}$ is called a \emph{chain through $y$}.

\begin{lem}\label{lem:class-count-monotonicity}
For every $s\in\Z$, $y\in\R$ and $n\in\Z^+$,
$$
\card(\C_{s,n}(y)) \leq \card(\C_{s,n+1}(y)).
$$
\end{lem}

\begin{proof}
For each $E\in\C_{s,n}(y)$, we have $y\in\Ra_s(E)$.
Define
$$
 \Omega(y,E) := \big\{ D\in\C_{s,n+1}(y) : D\text{ is an extension of }E \big\}.
$$
If $E\neq F$, then $\Omega(y,E) \cap \Omega(y,F) =\emptyset $.
Then it suffices to show $\Omega(y,E)$ is nonempty.

If $\sigma_s(E)\neq-1$, then Lemma~\ref{lem:class-range} gives
$$
  y \in \Ra_s(E) = \bigcup_{k\in\Sigma}\Ra_s\big(\{\bfw k:\bfw\in E\}\big).
$$
Then there exists $k\in\Sigma$ with $ y \in \Ra_s\big(\{\bfw k:\bfw\in E\}\big)$. Hence $\Omega(y,E)$ is nonempty.

If $\sigma_s(E)=-1$, Lemma~\ref{lem:class-range} gives $\Ra_s(E)=\Ra_s(\{\bfw k:\bfw\in E,\ k\in\Sigma\})$.
Then
$$
 	\Omega(y,E) = \big\{ \{\bfw k:\bfw\in E , k \in \Sigma \}   \big\}.
$$
Therefore
$$
 \card\big(\C_{s,n}(y)\big)
 = \sum_{E\in\C_{s,n}(y)}  1
 \leq \sum_{E\in\C_{s,n}(y)}  \card\big( \Omega(y,E) \big)
 \leq \card\big(\C_{s,n+1}(y)\big).
$$
Now the proof is complete.
\end{proof}



\begin{lem}\label{lem:chain-coding}
For every $y\in\R\setminus T_r(\B_r)$, the chains through $y$ are in
one-to-one correspondence with the local level sets contained in $L_r(y)$.
\end{lem}

\begin{proof}
Let $(C_n)_{n\in\Z^+}$ be a chain through $y$.  For each
$n\in\Z^+$, define
$$
 K_n := L_r(y) \cap \bigcup_{\bfw\in C_n}I_{\bfw}.
$$
Since $C_n\in\C_{0,n}(y)$, we have $y\in\Ra_0(C_n)$.
Hence $K_n$ is nonempty.
Since $T_r$ is continuous, $L_r(y)$ is closed.
Notice that $I_\bfw \subset I$ is a closed interval,
therefore $K_n$ is a compact subset of $I$.

Since $C_{n+1}$ is an extension of $C_n$, for every
$\bfw k\in C_{n+1}$ with $\bfw\in\Sigma^n$ and $k\in\Sigma$,
we have $\bfw\in C_n$.
Since
$I_{\bfw k}\subset I_{\bfw}$, we have $K_{n+1}\subset K_n$. Thus $(K_n)$ is a nested sequence of nonempty compact subsets of $I$,
so 
$$
 K := \bigcap_{n\in\Z^+}K_n\neq\emptyset.
$$

$y \notin T_r(\B_r)$ implies that $K \cap \B_r = \emptyset$.
Choose $ x\in K$ with $\bfw(x) = w_1 w_2 \cdots$.
Then we have $T_r(x)=y$. 
For every $n\in\Z^+$,
we have $  w_1 \cdots w_n \in C_n$ and hence $C_n(x)=C_n$.

Pick $x' \in K$, then $ C_n(x')=C_n(x)$ for every $n\in\Z^+$.  
Lemma~\ref{lem:class-lls-one} implies that $x'\sim_r x$. 
Thus each chain determines a unique local level set contained in $L_r(y)$.

Conversely, assume that $L_{r,u}^{\rm loc} \subset L_r(y)$.
Then we have $T_r(u)=y$ and $u\notin\B_r$.  
It follows from Lemma~\ref{lem:sim-symmetry} that $ u \sim_r (1-u) $.
Without loss of generality, we assume $u <1/2$ so that $u \in I\setminus\B_r$.
Assume that the word address of $u$ is given by $ \bfw(u)= w_1 w_2 \cdots $.

For every $n\in\Z^+$, since
$u\in I_{w_1 \cdots w_n}$ and $T_r(u)=y$, we have
$$
 y\in T_r (I_{w_1 \cdots w_n} ) =\Ra_0\big(C_n(u)\big).
$$
Hence $C_n(u)\in\C_{0,n}(y)$. 
Moreover, Lemma~\ref{lem:class-extension} shows that $C_{n+1}(u)$ is an extension of $C_n(u)$. 
Therefore $(C_n(u))_{n\in\Z^+}$ is a chain through $y$.

Given $ v \in L_{r,u}^{\rm loc} \cap I$, we have $v\sim_r u$.
Hence Lemma~\ref{lem:class-lls-one} gives $ C_n(v)=C_n(u)$ for every $n\in\Z^+$.
Assume that local level sets $ L_{r,u}^{\rm loc}$ and $ L_{r,t}^{\rm loc}$ correspond to the same chain.
Then $ C_n(u)=C_n(t)$
for every $n\in\Z^+$. By Lemma~\ref{lem:class-lls-one}, we obtain $u\sim_r t$.
Thus $ L_{r,u}^{\rm loc}= L_{r,t}^{\rm loc}$.  
The correspondence between local level sets and chains is therefore one-to-one.
\end{proof}

Here and below, an infinite count is denoted by $+\infty$.

\begin{thm}\label{thm:lls-count}
For every $y\in\R\setminus T_r(\B_r)$, we have
$$
 N_r^{\rm loc}(y) =\lim_{n\to\infty}\card(\C_{0,n}(y)).
$$
\end{thm}

\begin{proof}
For each $n\in\Z^+$, write
$b_n:=\card\big(\C_{0,n}(y)\big)$.
By Lemma~\ref{lem:class-count-monotonicity}, we have
$b_n\leq b_{n+1}$.

Fix $n\in\Z^+$. 
For each $E\in\C_{0,n}(y)$, let $\bfw\in E$ be arbitrary.
Since
$$
 y\in\Ra_0 (E)=T_r(I_{\bfw}),
$$
there exists $x \in I_{\bfw} \setminus \B_r$ with $T_r(x)=y$.
Hence we have $C_n(x)=E$.
If $C_n(x)\neq C_n(x')$, then Lemma~\ref{lem:class-lls-one} gives $x\not\sim_r x'$,
and hence $L_{r,x}^{\rm loc}\neq L_{r,x'}^{\rm loc}$.
Therefore
$$
 N_r^{\rm loc}(y)\geq\card\big(\C_{0,n}(y)\big)=b_n.
$$
In particular, if $\{ b_n \}$ is unbounded, the desired equality follows immediately.

It remains to consider the case that $\{ b_n \}$ is bounded.
Then there exists $N\in\Z^+$ such that $b_n=b_N$ for every $n \geq N$.
Using the notation in Lemma~\ref{lem:class-count-monotonicity}, for any $n \geq N$, 
$$
 b_n
 = \sum_{E\in\C_{0,n}(y)}1
 \leq \sum_{E\in\C_{0,n}(y)}\card\big(\Omega(y,E)\big)
 \leq b_{n+1}
 = b_n.
$$
It follows that $ \card\big(\Omega(y,E)\big)=1$ for every $n\geq N$ and $E\in\C_{0,n}(y)$.
Consequently, for every $E\in\C_{0,N}(y)$, there is exactly one chain $(C_n)_{n\in\Z^+}$ through $y$ with $C_{N}=E$.
By Lemma~\ref{lem:chain-coding}, each local level set contained in $L_r(y)$ corresponds to a chain.
Hence
$$
 N_r^{\rm loc}(y) \leq \card\big(\C_{0,N}(y)\big) = b_N = \lim_{n\to\infty} b_n.
$$
Together with the opposite inequality proved above, this completes the proof.
\end{proof}

\section{The expected number of local level sets}
We now integrate the counting functions $\card(\C_{s,n}(y))$
to compute the expected number of local level sets. In this section, we denote $p:= \lceil r/2 \rceil,$ and $q:= \lfloor r/2 \rfloor.$
Thus $p+q=r$.

For every interval $J\subset \R$, denote its length by $|J|$.  For any $s\in\Z$ and $n\geq0$, define
$$
 A_{s,n}
 :=\sum_{C\in\C_{s,n}}|\Ra_s(C)|
 =\int_\R\card(\C_{s,n}(y))\,dy.
$$
In particular, $ A_{s,0} = |\Ra_s( \{ \vartheta \} )|=|F_s(I)|$.

\begin{lem}
For every $s\in\Z$ and $n\geq0$,
\begin{equation}\label{eq:cal-asn}
 rA_{s,n+1}=
 \begin{cases}
  pA_{s+1,n}+qA_{-s-1,n},&s\neq-1,\\
  A_{0,n},&s=-1.
 \end{cases}
\end{equation}
\end{lem}

\begin{proof}
Suppose that $s\neq-1$. For any $k\in\Sigma$ and
$C\in\C_{(-1)^k(s+1),n}$, Lemma~\ref{lem:relation-fs} gives
$$
 r\big|\Ra_s \big( \{k\bfw:\bfw\in C\} \big) \big|
 =\big| \Ra_{(-1)^k(s+1)}(C) \big|.
$$
Summing according to Lemma~\ref{lem:class-prefix}, we obtain
$$
 rA_{s,n+1}
 =\sum_{k\in\Sigma}A_{(-1)^k(s+1),n}
 =pA_{s+1,n}+qA_{-s-1,n}.
$$

Now we assume $s=-1$. 
For any $C\in\C_{0,n}$, \eqref{eq:cal-fs-0} gives
$$
 r\big|\Ra_{-1} \big(\{k\bfw:k\in\Sigma,\ \bfw\in C\} \big) \big| =\big| \Ra_0(C) \big|  .
$$
Lemma~\ref{lem:class-prefix} also yields
$$
 rA_{-1,n+1}
 =\sum_{C\in\C_{0,n}} \big| \Ra_0(C) \big|
 =A_{0,n}.
$$
\end{proof}

By Lemma~\ref{lem:class-count-monotonicity}, $A_{s,n}$ is nondecreasing for
$n\geq1$.  Also,
$$
 F_s(I)=\bigcup_{k\in\Sigma}F_s \big( S_k(I) \big),
$$
which gives $A_{s,0}\leq A_{s,1}$.  
Hence the limit $ A_s:=\lim_{n\to\infty}A_{s,n}$ exists in $[0,+\infty]$.

To estimate $A_s$, we use the following value of
$M_r:=\max_{x\in[0,1]}T_r(x)$, due to Baba~\cite{B84}.

\begin{lem}\label{lem:max-Baba}
If $r$ is even, then $M_r=r^2/(2r^2-2)$. 
If $r$ is odd, then $M_r=r/(2r-2)$. 
\end{lem}

\begin{lem}\label{lem:cal-as}
For every $s\in\Z$, we have $A_s<\infty$ and $|A_s-|s|/2|\leq3M_r$.
\end{lem}

\begin{proof}
It is clear that $ T_r(I)=T_r([0,1])=[0,M_r]$.
Then for any $u,v\in I$,
$$
 |F_s(u)-F_s(v)|
 \leq |T_r(u)-T_r(v)|+|s||u-v|
 \leq M_r+\frac{|s|}{2}.
$$
Since $F_s(I)$ is an interval with $F_s(0)=0$,
\begin{equation}\label{eq:as-cal-0}
 \frac{|s|}{2}-M_r
 \leq \Big|T_r\Big(\frac{1}{2}\Big)+\frac{s}{2}\Big|
 \leq A_{s,0}
 \leq M_r+\frac{|s|}{2}.
\end{equation}
Hence $ A_s\geq A_{s,0}\geq |s|/{2}-M_r$.

For the upper bound, we prove by induction that, for every $n\geq0$,
\begin{equation}\label{eq:asn-cal-upper}
 A_{s,n}\leq
 \begin{cases}
  M_r+\frac{r}{2q}+\frac{s}{2},&s\geq0,\\[6pt]
  M_r-\frac{s}{2},&s\leq-1.
 \end{cases}
\end{equation}
The case $n=0$ follows from \eqref{eq:as-cal-0}. 
Now we assume \eqref{eq:asn-cal-upper} holds for $n$. 
If $s \geq 0$,
$$
 A_{s,n+1}
 \leq\frac{p}{r}\Big( M_r + \frac{r}{2q} + \frac{s+1}{2}\Big)
 +\frac{q}{r} \Big(M_r + \frac{s+1}{2}\Big)
 =M_r + \frac{r}{2q} + \frac{s}{2}
 \leq 3 M_r + \frac{s}{2}.
$$
If $s=-1$, $ A_{-1,n+1} = {A_{0,n}}/{r} \leq  M_r+1/2$.
If $s \leq -2$, then
$$
A_{s,n+1}
 \leq\frac{p}{r}\Big(M_r-\frac{s+1}{2}\Big)
 +\frac{q}{r}\Big(M_r+\frac{r}{2q}+\frac{-s-1}{2}\Big)
 =M_r-\frac{s}{2}.
$$
By \eqref{eq:asn-cal-upper}, we have
$
 A_s = \sup_{n \in \Z^+} A_{s,n} \leq 3M_r +|s|/2.
$
Then the proof is complete.
\end{proof}

\begin{lem}\label{lem:A0-cal}
We have
$$
 A_0=\frac{pr}{2q(r-1)}.
$$
\end{lem}

\begin{proof}

By Lemma~\ref{lem:cal-as}, all $A_s$ are finite.
Letting $s=-1$ and $n$ tend to $\infty$ in \eqref{eq:cal-asn} gives $ rA_{-1}=A_0$.
Then for any $  s\neq-1$,
$$
 rA_s=  pA_{s+1}+qA_{-s-1}.
$$
Therefore we obtain $rA_0=pA_1+qA_{-1}$.
Since $p+q=r$, we have
$$
 p(A_1-A_{-1})=(r-1)A_0.
$$
For every $s\in\Z^+$, the same recursion gives
$$
 p(A_{s+1}-A_{-s-1})
 =r(A_s-A_{-s-1})
 =p(A_s-A_{-s}).
$$
Thus $A_s-A_{-s}=A_1-A_{-1}$ for every $s\in\Z^+$.
Consequently, for every $s\geq0$,
$$
 r(A_{s+1}-A_s)
 =q(A_{s+1}-A_{-s-1})
 =q(A_1-A_{-1})
 =\frac{q(r-1)}{p}A_0.
$$
Summing from $0$ to $s-1$, we obtain
$$
A_s = A_0 + \frac{q(r-1) s}{pr}A_0.
$$
Lemma~\ref{lem:cal-as} yields $\lim_{s \to \infty} A_s/s=1/2$. Hence
$$
	A_0 = \lim_{s \to \infty} \Big( 1 +\frac{q(r-1) s}{pr} \Big)^{-1} A_s 
	=   \lim_{s \to \infty}  \frac{pr}{ q(r-1) + pr /s}  \cdot      \frac{A_s}{s}
	=  \frac{pr}{ 2 q(r-1) } .
$$
This completes the proof.
\end{proof}

\begin{proof}[Proof of Theorem~\ref{thm:main}]
By \eqref{eq:def-br}, $\B_r$ is a countable union of finite sets.
Hence $T_r(\B_r)$ is countable and has Lebesgue measure zero.  By
Theorem~\ref{thm:lls-count} and the monotone convergence theorem,
$$
 \int_\R N_r^{\rm loc}(u)\,du
 =\lim_{n\to\infty}\int_\R\card(\C_{0,n}(u))\,du
 =A_0.
$$
Therefore
$$
 \E \big[N_r^{\rm loc}(y)\big]
 =\frac{1}{M_r}\int_0^{M_r}N_r^{\rm loc}(u)\,du
 =\frac{A_0}{M_r}.
$$
If $r$ is even, Lemmas~\ref{lem:max-Baba} and \ref{lem:A0-cal} give
$$
\E \big[N_r^{\rm loc}(y)\big]
= \frac{A_0}{M_r}
 =\frac{r}{2(r-1)} \cdot  \frac{2(r^2-1)}{r^2}
 =\frac{r+1}{r}.
$$
If $r$ is odd, Lemmas~\ref{lem:max-Baba} and \ref{lem:A0-cal} give
$$
\E \big[N_r^{\rm loc}(y)\big]
= \frac{A_0}{M_r}
 =\frac{r(r+1)}{2(r-1)^2}  \cdot \frac{2(r-1)}{r}
 =\frac{r+1}{r-1}.
$$
\end{proof}

\subsection*{Acknowledgements}

The authors would like to thank Shi-Lei Kong, Yan-Qi Qiu and Huo-Jun Ruan for helpful discussions. 
Ying and Zhang were financially supported by Zhejiang University Student Research Training Program No. 2024R401135.

\bibliographystyle{amsplain}

\end{document}